\documentclass[conference]{IEEEtran}
\usepackage{cite}

\ifCLASSINFOpdf
\else
\fi
\usepackage{amsmath}
\usepackage{amsfonts}
\usepackage{amsthm}
\newtheorem{theorem}{Theorem}

\newtheorem{rem*}{Remark}
\newtheorem{question}{Question}

\DeclareMathOperator{\TR}{tr}
\DeclareMathOperator{\SPAN}{span}

\newcommand{\UNFs}[1]{\overline{\Omega}_{n,m} {(\mathbb {#1})} } 
\newcommand{\UNTFs}[1]{\Omega_{n,m} {(\mathbb {#1})} } 
\newcommand{\Gr}[1]{\mu_{n,m} (\mathbb #1)} 
\newcommand{\Grbar}[1]{\overline{\mu}_{n,m} (\mathbb #1)} 

\begin{document}
%
\title{On the structures\\ of Grassmannian frames}

\author{\IEEEauthorblockN{John I. Haas IV}
\IEEEauthorblockA{Department of Mathematics\\
University of Missouri\\
Columbia, Missouri 65211\\
http://faculty.missouri.edu/$\sim$haasji/\\
haasji@missoui.edu
}
and\\
\IEEEauthorblockN{Peter G. Casazza}
\IEEEauthorblockA{Department of Mathematics\\
University of Missouri\\
Columbia, Missouri 65211\\
casazzap@missoui.edu
}
}


\maketitle

\begin{abstract}
A common criterion in the design of finite Hilbert space frames is minimal coherence, as this leads to error reduction in various signal processing applications. Frames that achieve minimal coherence relative to all unit-norm frames are called {\bf Grassmannian frames}, a class which includes the well-known equiangular tight frames. However, the notion of ``coherence minimization'' varies according to the constraints of the ambient optimization problem, so there are other types of ``minimally coherent'' frames one can speak of. 

In addition to Grassmannian frames, we consider the class of frames which minimize coherence over the space of frames which are both unit-norm and tight, which we call {\bf $1$-Grassmannian frames}. We observe that these two types of frames coincide in many settings, but not all; accordingly, we investigate some of the differences between the resulting theories.  For example, one noteworthy advantage of $1$-Grassmannian frames is that their optimality properties are preserved under the Naimark complement.

\end{abstract}


%
\IEEEpeerreviewmaketitle

\section{Introduction}

Interest in the theory of Hilbert space frame theory has grown rapidly over the last two decades~\cite{MR836025, MR2892742}. While the reasons for this recent expansion are numerous and diverse~\cite{MR2105392, }, the quintessential application of finite frames involves the analysis of signals, for example, quantization~\cite{MR1486646}, reconstructing an image sent from a damaged transmitter~\cite{MR2021601}, or determining quantum states~\cite{Zau, MR2662471}.  

This note concerns a certain quantitative frame property, called {\it coherence}.  It is known~\cite{MR2021601, MR1984549} that frames with minimal coherence are often the ``best'' models for signal reconstruction apparatuses, as they are optimally robust against certain forms of noise and erasures.

We consider two disparate notions of {\it optimally incoherent} frames:  the familiar {\bf Grassmannian frames} and the so-called {\bf 1-Grassmannian frames}, which we introduce shortly and which have previously been considered in \cite{haas_phd, 2015arXiv150308690G}.  We study examples arising from both definitions, noting that they coincide in many cases but may manifest with notably different qualitative and quanitative properties in other situations. In addition, we note that $1$-Grassmannian frames enjoy a property that, in general, normal Grassmannian frames do not: invariance under the Naimark complement (see Theorem~\ref{thm_naim}).

\subsection{Frame Basics}
Let $\{e_j\}_{j=1}^m$ denote the canonical orthonormal basis for $\mathbb F^m$, where $\mathbb F = \mathbb R$ or $\mathbb C$ and let $I_m$ denote the $m \times m$ identity matrix. A set of vectors $\Phi = \{\phi_j\}_{j=1}^n \subset \mathbb F^m$ is a {\bf (finite) frame} if $\SPAN \{\phi_j\}_{j=1}^n = \mathbb F^m.$  
It is convenient to identify a frame  $\Phi = \{\phi_j\}_{j=1}^N$ with its {\bf synthesis matrix}
$$\Phi = \left[ \phi_1 \, \, \phi_2 \, \, ... \,\, \phi_n \right],$$
the $m \times n$ matrix with columns given by the {\bf frame vectors}.  Just as we have written
$\Phi = \{\phi_j\}_{j=1}^n $
and 
$\Phi = \left[ \phi_1 \, \, \phi_2 \, \, ... \,\, \phi_n \right]$ in the last sentence, we  continue with the tacit understanding that $\Phi$ represents both a matrix and a set of vectors.  Furthermore, we reserve the symbols $m$ and $n$ to refer to the dimension of the span of a frame and the cardinality of a frame, respectively.

A frame $\Phi = \{\phi_j\}_{j=1}^n$  is {\bf $a$-tight} if 
$$\Phi \Phi^* = \sum_{j=1}^n \phi_j  \phi_j^*= a I_m, \text{ for some } a>0$$ and it is {\bf unit-norm} if each frame vector has norm $\|\phi_j\|=1$. 
 If $\Phi$ is unit-norm and $a$-tight, then $a=\frac{n}{m}$ because
\begin{align*}
n &= \sum_{l=1}^m\sum_{j=1}^n |\langle e_l, \phi_j \rangle|^2 
\\
&= \sum_{l=1}^m\sum_{j=1}^n \TR(\phi_j \phi_j^* e_l e_l^*) = a \sum_{l=1}^m \| e_l \|^2 = am,
\end{align*}
in which case also we have the identity
\begin{equation}\label{eq_tight_un}
\sum\limits_{l=1}^n | \langle \phi_j, \phi_l \rangle |^2  = \frac{n}{m}, \text{ for every }  j \in \{1,2,...,n \}.
\end{equation}
%
%

Let $\UNFs{F}$ denote the space of unit-norm frames for $\mathbb F^m$  consisting of $n$ vectors  and let $\UNTFs{F}$ denote the space of unit-norm, tight frames for $\mathbb F^m$ consisting of $n$ vectors.
Given any set of unit vectors, $\Phi = \{ \phi_j \}_{j=1}^n \in \mathbb F^m$,  its {\bf coherence} is defined by
$$\mu(\Phi) = \max\limits_{j \neq l} |\langle \phi_j, \phi_l \rangle |.$$
We define and denote the {\bf Grassmannian constant} as
$$
\Grbar{F} = \min\limits_{\scriptscriptstyle \Phi \in \UNFs{F}} \mu(\Phi)
$$ 
and the {\bf $1$-Grassmannian constant} as
$$
\Gr{F} = \min\limits_{\scriptscriptstyle \Phi \in \UNTFs{F}} \mu(\Phi).
$$ 
Correspondingly, a frame $\Phi \in \UNFs{F}$ is a {\bf Grassmannian frame} if
$$
\mu(\Phi) = \Grbar{F}. 
$$
and a frame $\Phi \in \UNTFs{F}$ is a {\bf $1$-Grassmannian frame} if
$$
\mu(\Phi) = \Gr{F}.
$$

\begin{rem*}
A cautious reader might question the well-definedness of the preceding definitions - particularly whether or not a minimally coherent configuration of $n$ unit vectors, in fact, spans $\mathbb F^m$, thereby forming a frame; in the following, we sketch that, if a minimally coherent set of unit vectors exists that does not span $\mathbb F^m$,  then it can be ``rearranged'' to form a frame with the same coherence. 

Note that $\UNFs{F} \subset \left[\mathcal S_m(\mathbb F)\right]^n$, where $\left[\mathcal S_m(\mathbb F)\right]^n$ denotes the $n$-fold Cartesian product of unit spheres in $\mathbb F^m$, which is a compact set.  It follows by the extreme value theorem that a minimizer, $\Phi=\{\phi_j\}_{j=1}^n \in \left[\mathcal S_m(\mathbb F)\right]^n$, for the coherence function exists.  If $n=m$, then $\Phi$ must be an orthonormal basis and if $n=m+1$, then it is well-known~\cite{MR2964005} that $\Phi$ corresponds to an $m$-simplex (in particular, it spans), so assume that $n \geq m+2$, let
$
\mathcal H := \SPAN \{\phi_j\}_{j=1}^n
$ and suppose that $\dim(\mathcal H) =m-k$ 
 for some positive integer, $k<m$.  

In this case, fix any pair of vectors, $\phi_j$ and $\phi_l$, such that 
$
\left| \langle \phi_j, \phi_l \rangle \right|
$
equals the minimal coherence value, fix any $m-k$ vectors that form a basis for the subspace, $\mathcal H$, and  replace any choice of $k$ vectors from the remaining unfixed vectors with an orthonormal basis for $\mathcal H^\perp$.  The newly formed set of vectors retains the minimal coherence and spans $\mathbb F^m$, which shows that $\Grbar{F}$ is well-defined and that Grassmannian frames exist for all positive integers, $m \leq n$.

The well-definedness of $\Gr{F}$ and the corresponding existence of $1$-Grassmannian frames is more straightforward, because, indeed,  $\UNTFs{F}$ is a compact set, as it is the intersection of $\left[\mathcal S_m(\mathbb F)\right]^n$ with the corresponding Stiefel manifold, which is well-known to be compact~\cite{MR2364186}.
\end{rem*}

\section{Two notions of optimal incoherence}
The term ``Grassmannian frame'' was coined in \cite{MR1984549} as an acknowledgment of the connection with the  Grassmannian line packing problem.  Since then, the study of Grassmannian frames has manifested almost exclusively as the study of {\bf equiangular tight frames (ETFs)}~\cite{MR2410113, 5946867, Fickus:2015aa}, which are unit-norm, tight frames for which the set of pairwise absolute inner products between frame vectors is a singleton. Concerted research into this particular class of Grassmannian frames is largely attributed to the following theorem of Welch~\cite{welch:bound}, adapted for our frame-theoretic setting.

\begin{theorem}\label{thm_wel}(Welch, \cite{welch:bound}; see also  \cite{Fickus:2015aa})
\\
Given a unit-norm frame $\Phi=\{\phi_j\}_{j=1}^n \in \UNFs{F}$,
then
$$
\mu(\Phi) \geq \Grbar{F} \geq W_{n,m},
$$
where $W_{n,m} = \sqrt{\frac{n-m}{m(n-1)}}$ (the {\bf Welch constant}).
\end{theorem}


  Given the goal of constructing a Grassmannian frame in $\UNFs{F}$, the assumption (or hope) that it might manifest as an ETF is appealing because then Theorem~\ref{thm_wel} provides three highly rigid construction principles. {\bf (i)} Only one value may occur among the set of pairwise absolute inner products, {\bf (ii)} this value must be the Welch constant and {\bf (iii)} the frame must be tight.

We remark that Principle~{\bf (ii)} is perhaps the greatest motivation for study of ETFs.  In most cases, little is known about the value of the Grassmannian constant.  Without an analytic expression for $\Grbar{F}$, it is difficult to determine whether a frame with seemingly low coherence is, in fact, a Grassmannian frame.

However, Principle {\bf (iii)} is most relevant for our current narrative,  because it implies that a Grassmannian frame with coherence equal to the Welch constant must be tight.
Conversely, the set containment $\UNTFs{F} \subset \UNFs{F}$ implies $\Gr{F} \geq \Grbar{F}$. Together, these observations show that every ETF is simultaneously a Grassmannian frame and a $1$-Grassmannian frame. A simple generalization of these observations yields the following. 

\begin{theorem}\label{thm_grvsgr1}
If $\Phi$ is a tight Grassmannian frame, then $\Phi$ is a $1$-Grassmannian frame.
\end{theorem}

Although Theorem~\ref{thm_grvsgr1} is obvious, our main thesis lies with the falsehood of its converse. As we demonstrate below, there are cases where no tight Grassmannian frames exist.  Within the context of signal analysis, this means we ocasionally face the problem of trading the ``perfect reconstruction'' endowed by tightness with the ``robustness against error'' endowed by low coherence.   Thus, we think of $1$-Grassmannian frames as the ideal reconciliation of these two issues: they minimize coherence without sacrificing tightness.

Although $1$-Grassmannian frames have been considered before~\cite{haas_phd, 2015arXiv150308690G}, we believe that they have received little overall attention so far.  Therefore, for the remainder of the paper, we use known examples to highlight the peculiar relationship between Grassmannian frames and their counterparts, $1$-Grassmannian frames. We conclude by observing that the useful Naimark complement principle holds for $1$-Grassmannian frames and apply it to obtain an answer to Question~\ref{quest1} below.

\subsection{On the falsity of the converse of Theorem~\ref{thm_grvsgr1}}
 To justify the claim that the study of $1$-Grassmannian frames is inequivalent to the study of Grassmannaian frames, we disprove the converse of Theorem~\ref{thm_grvsgr1}. Fortunately, there is little work to be done, as the needed counterexample is already provided in \cite{BK06} where the authors studied certain Grassmannian frames in $\mathbb R^2$ and $\mathbb R^3$.

 It well-known~\cite{Fickus:2015aa, MR0023530} that if $\Phi  \in \overline \Omega_{6,3}(\mathbb R)$ corresponds to six non-antipodal vertices from a regular icosahedron, then $\Phi$ is an ETF. Given such a $\Phi$, then by our previous observations, $\Phi$ is both  a Grassmannian frame and a $1$-Grassmannian frame, and 
$$\mu(\Phi)=\mu_{6,3}  (\mathbb R) =\overline{\mu}_{6,3} (\mathbb R) =W_{6,3}=\frac{1}{\sqrt 5}.$$
Using a combination of analytic and exhaustive methods, the authors of \cite{BK06} deduced the following two key observations.
\\  
{\it Observation (i)} The Grassmannian constant for five unit vectors in $\mathbb R^3$ is the same as six vectors in $\mathbb R^3$.  That is,
$$
\overline\mu_{5,3} (\mathbb R) = \overline\mu_{6,3} (\mathbb R)= \frac{1}{\sqrt 5}.
$$ 
{\it Observation (ii)} There are no tight frames consisting of five unit vectors in $\mathbb R^3$ with coherence that achieves this constant.

Hence,  discarding a single vector from $\Phi$ yields a new Grassmannian frame; however,  Observation~{(ii)} shows that a Grassmannian frame is not tight and, in fact, that any tight frame in $\mathbb R^3$ composed of five unit vectors must have strictly greater coherence. Ipso facto, the authors~\cite{BK06} had proven that $\overline\mu_{5,3} (\mathbb R)$ is stictly less that $\mu_{5,3} (\mathbb R) $, thereby demonstrating falsehood for the converse of Theorem~\ref{thm_grvsgr1}.
 
Within the context of this note,  a natural question emerges from their observations.
\begin{question}\label{quest1}
Since $\mu_{5,3} (\mathbb R) > \overline\mu_{5,3} (\mathbb R) = \frac{1}{\sqrt 5}$, what is the value $\mu_{5,3} (\mathbb R)$?
\end{question}
We will use the {\it Naimark complement} to answer this question in the final section.

\subsection{Interplay between the two notions of ``low coherence''}
Although Grassmannian frames and $1$-Grassmannian frames are, in general, different objects, it is surprising how often they coincide.  Because it occurs so often, we say that a Grassmannian frame is a {\bf universal Grassmannian frame} if it is also a $1$-Grassmannian frame. We list a few known infinite families of universal Grassmannian frames.

({\it ETFs}) Not surprisingly, the first family of universal Grassmannian frames we mention are the ETFs.  Although this class of Grassmannian frames remains shrouded with open questions~\cite{Fickus:2015aa},  numerous infinite families of ETFs have been constructed~\cite{MR2235693, fickus2016tremain, MR2890902, Boumed_2014, MR2639246, Fickus:2015aa} over the last few decades.  As we have disussed, ETFs are universal Grassmannian frames.

({\it OGFs}) Other infinite families of universal Grassmannian frames come as so-called {\bf orthoplectic Grassmannian frames (OGFs)}, including {\it maximal sets of mutually unbiased bases} \cite{MR2778089} and two other families based on relative difference set constructions from \cite{bodmann2015achieving}.  An {\bf orthoplectic Grassmannian frame (OGF)} is a Grassmannian frame with sufficiently large cardinality and with coherence equal to the {\bf orthoplex constant}, as defined in the following theorem.

\begin{theorem}\label{thm_orth}({\it Orthoplex bound}, \cite{MR0074013, MR1418961})
\\
If $\mathbb F=\mathbb R$, let $d=\frac{m(m+1)}{2}$, or if $\mathbb F=\mathbb C$, let $d=m^2$.
\\
Given a unit-norm frame $\Phi\in\UNFs{F}$ with $n~>~d$,
then
$$
\mu(\Phi) \geq \Grbar{F} \geq O_{n,m},
$$
where $O_{n,m} = \frac{1}{\sqrt{m}}$ (the {\bf orthoplex constant}).  
\end{theorem}

Unlike the characterization of ETFs given by the Welch bound in Theorem~\ref{thm_wel}, Theorem~\ref{thm_orth} does not imply that an OGF is necessarily tight.  In other words, every OGF is a Grassmannian frame but not necessarily a $1$-Grassmannian frame.   Indeed, for most OGFs belonging to the families of universal Grassmannian frames just mentioned, it is possible to delete several vectors  such that the remaining subset is still an OGF~\cite{bodmann2015achieving}; however, the tightness property is almost always lost in these cases.  This leads to several questions, but we focus on one.

\begin{question}
Given the existence of an OGF $\Phi \in \Omega_{n,m}(\mathbb F)$, under what conditions does $\Gr{F}=\Grbar{F}$?
\end{question}

Although we are far from a complete answer to this question, an example from a recent work~\cite{2016arXiv160502012C} yields  a surprising insight, which we interpret as a partial answer.  In \cite{2016arXiv160502012C}, the authors construct two distinct orthoplectic Grassmannian frames $\Phi, \Psi \in \overline \Omega_{5,2}(\mathbb C)$, where $\Phi$ is tight while $\Psi$ is not.  Thus, $\Phi$ is a univeral Grassmannian frame while is $\Psi$ not.  Moreover, the cardinality of the {\bf angle set} --- or set of absolute pairwise inner products --- for $\Phi$ differs from that of $\Psi$.  Thus, there are instances where $\Gr{F} = \Grbar{F}$, but Grassmannian frames may manifest with vastly different geometric and spectral properties.  In particular, the coexistence of a universal Grassmannian frame with a nonuniversal Grassmannian frame is possible.

\section{The Naimark complement}
A nice aspect from the theory of tight frames is that many geometric properties are preserved under the so-called {\it Naimark complement} technique \cite{MR2964005, MR3062405}.  For us, this manifests as a statement about  the $1$-Grassmannian constant.

\begin{theorem}\label{thm_naim}({\it Naimark complement})
If a $1$-Grassmannian frame $\Phi \in \UNTFs{F}$ has coherence $\mu_{n,m}(\mathbb F)$, then a $1$-Grassmannian frame $\Phi' \in \Omega{\scriptstyle n,n - m}({\mathbb F})$ exists, and its coherence is $\frac{n-m}{m} \Gr{F}$.
More succinctly,
$$
\mu{\scriptstyle n,n - m}({\mathbb F}) = \frac{n-m}{m}\mu{\scriptstyle n,m}({\mathbb F}).
$$
\end{theorem}

\begin{proof} We sketch the proof.
Because $\Phi$ is tight, it is easily checked that $G = \frac{m}{n}\Phi^* \Phi$ is an orthogonal projection onto an $m$-dimensional subspace of $\mathbb F^n$.  Moreover, $G$ is the so-called {\bf Gramian} of $\Phi$, meaning that its entries encode the  inner products between the frame vectors of $\Phi$.  Consider the projection $I_n-G$, which projects onto the $(n-m)$-dimensional subspace orthogonal to the range of $G$.  Since the largest absolute value among $G$'s off-diagonal entries is $\frac{m}{n}\mu_{n,m} (\mathbb F)$, it follows that the largest absolute value among the off-diagonal entries of $I_n-G$ is also  $\frac{m}{n}\mu_{n,m} (\mathbb F)$.  By applying the singular value decomposition to $I_n-G$, exploiting our identification of frames with matrices, and rescaling, it follows that a tight frame $\Psi \in \Omega{\scriptstyle n,n-m}(\mathbb F)$ exists such that $\frac{n-m}{n}\Psi^* \Psi = I_n-G$.
Thus, the coherence of $\Psi$ is computed by rescaling,
$$
\mu(\Psi) 
= \frac{n}{n-m}\cdot \frac{m}{n}\mu_{n,m}({\mathbb F})
=
\frac{m}{n-m}\mu_{n,m}({\mathbb F}).$$  
By way of contradiction, if 
$$\mu{\scriptstyle n,n - m}({\mathbb F}) < \frac{m}{n-m}\mu_{n,m}(\mathbb F),$$
then reversing this argument would yield a tight frame $\Phi' \in \UNTFs{F}$ with coherence less that $\Gr{F}$, contradicting our assumption that $\Gr{F}$ is the $1$-Grassmannian constant for this setting.
\end{proof}

This leads to an answer to Question~\ref{quest1}.
Incidentally, we need another result from \cite{BK06}, where they showed that the five vertices of a regular pentagon form a tight Grassmannian frame in $\mathbb R^2$.  In particular,
$$
\mu_{5,2}({\mathbb R})=\overline \mu_{5,2}({\mathbb R})= \cos( \pi/5),
$$
and the frame obtained from a regular pentagon is a universal Grassmannian frame in $\Omega_{5,2}(\mathbb F)$.  Using this along with the Naimark complement principle from Theorem~\ref{thm_naim}, we conclude that
$$
\mu_{5,3} (\mathbb R) = \frac{5-3}{3} \mu_{5,2}(\mathbb R) = \frac 2 3 \cos( \pi/5).
$$

As final a remark,  this example demonstrates that  it is possible for the Naimark complement of a non-universal 1-Grassmannian frame to be a universal Grassmannian frame.

\section*{Acknowledgment}

The authors were supported by
 NSF DMS 1307685;  NSF ATD 1321779; and ARO W911NF-16-1-0008.

\bibliographystyle{IEEEtran}
%
%

\bibliographystyle{IEEEtran}
\bibliography{sampta}
%
%
%

\end{document}